\documentclass[twocolumn]{autart}    
\usepackage{graphics} 

\usepackage{epsfig} 
\usepackage{amsmath} 
\usepackage{amssymb}  
\usepackage{bm}
\usepackage{booktabs}
\usepackage{url}
\usepackage{graphicx}
\usepackage{subfig}
\usepackage{tabularx}
\usepackage{multirow}
\usepackage{multicol}
\usepackage{amsfonts}
\usepackage{cite}
\usepackage{graphicx}
\usepackage{footnote}
\usepackage{graphics} 
\usepackage{epsfig} 
\usepackage{amsmath} 
\usepackage{amssymb}  
\usepackage{bm}
\usepackage{booktabs}
\usepackage{url}
\usepackage{graphicx}
\usepackage{subfig}
\usepackage{tabularx}
\usepackage{multirow}
\usepackage{multicol}
\usepackage{amsmath}
\usepackage{graphicx}
\usepackage{cite}
\usepackage{picinpar}
\usepackage{amssymb}
\usepackage{amsmath}
\usepackage{amsmath,bm}
\usepackage{stfloats}
\usepackage{url}
\usepackage{flushend}
\usepackage{colortbl}
\usepackage{soul}
\usepackage{multirow}
\usepackage{pifont}
\usepackage{color}
\usepackage{alltt}
\usepackage{enumerate}

\usepackage{epstopdf}
\usepackage{pbox}
\usepackage{caption}
\usepackage{booktabs}
\usepackage{amsmath}
\newtheorem{assumption}{Assumption}
\newtheorem{theorem}{\bf{Theorem}}
\newenvironment{proof}{\noindent{\hskip 2em \textbf{Proof: }}}{\hfill $ \blacksquare $ \vskip 4mm}

\usepackage{graphicx}          

\begin{document}

\begin{frontmatter}

\title{Nonlinear Cooperative Control of Double Drone-Bar Transportation System\thanksref{footnoteinfo}} 

\thanks[footnoteinfo]{This work is supported in part by National Natural Science Foundation of China under Grant 61873132 and Grant 61903200, in part by Natural Science Foundation of Tianjin under Grant 16JCZDJC30300 and Grant 19JCQNJC03500, in part by the China Postdoctoral Science Foundation under Grant 2020M670632, and in part by the Fundamental Research Funds for the Central Universities Nankai University under Grant 63201194.}

\author{Peng~Zhang}\ead{zhangpeng@mail.nankai.edu.cn},    
\author{Yongchun Fang\thanksref{footnoteinfo}}\ead{fangyc@nankai.edu.cn},                
\author{Xiao~Liang}\ead{liangx@nankai.edu.cn},  
\author{He~Lin}\ead{linhe@mail.nankai.edu.cn}, 
\author{Wei~He}\ead{hewei@mail.nankai.edu.cn} 

\thanks[footnoteinfo0]{Corresponding author at Institute of Robotics and Automatic Information System, Nankai University, Tianjin 300353, China. Tel.: +86 22 23505706; fax: +86 22 23500172.}

\address{Institute of Robotics and Automatic Information System,College of Artificial Intelligence\\ Tianjin Key Laboratory of Intelligent Robotics, Nankai University, Tianjin 300350, China}  


\begin{abstract}                          
Due to the limitation of the drone's load capacity, various specific tasks need to be accomplished by multiple drones in collaboration. In some transportation tasks, two drones are required to lift the load together, which brings even more significant challenges to the control problem because the transportation system is underactuated and it contains very complex dynamic coupling. When transporting bar-shaped objects, the load's attitude, the rope's swing motion, as well as the distance between the drones, should be carefully considered to ensure the security of the system. So far, few works have been implemented for double drone transportation systems to guarantee their transportation performance, especially in the aforementioned aspect. In this paper, a nonlinear cooperative control method is proposed, with both rigorous stability analysis and experimental results demonstrating its great performance. Without the need to distinguish the identities  between the leader and the follower, the proposed method successfully realizes effective control for the two drones separately, mainly owning to the deep analysis for the system dynamics and the elaborate design for the control law. By utilizing Lyapunov techniques, the proposed controller achieves simultaneous positioning and mutual distance control of the drones, meanwhile, it efficiently eliminates the swing of the load. Flight experiments are presented to demonstrate the performance of the proposed nonlinear cooperative control strategy.
\end{abstract}

\end{frontmatter}

\section{Introduction}
Drones are now widely used for material transportation because of their ability to take off and land vertically \cite{uav1,uav2,uav3,uav4,uav5,uav6,uav7,uav8,uav9,uav10}. Aerial transportation, which is not limited to terrain, plays an increasingly important role in emergency cases \cite{cheap1,cheap2}. Therefore, relevant research have always been concerned. When exploring drone delivery, researchers mainly adopt two methods, one of which is to use clamping claws or mechanical arms for transportation \cite{grasp1,grasp2,grasp3,grasp4}. In this mode, the attitude of the load is highly controllable, which makes it convenient for obstacles avoidance. Another way of transportation is to suspend the load by sling ropes \cite{cable1,cable4}, which ensures the flexibility of the drone itself and reduces additional cost on claws and arms. However, the load suspended by the cable inevitably sways under the drones, which may bring hidden danger to system security. For some particular transportation tasks, such as medical supplies, the damage caused by the sway of the load is unacceptable. Besides, a single drone usually does not have the ability to lift a heavy or a large-sized load. To solve these problems, it is a very natural idea to utilize two or even more drones to complete the tasks by cooperative control. Noting that effective control of the drone itself is already a very challenging work because the drone is underactuated with fewer control inputs than its degrees of freedom (DOFs), cooperative control for the double-drone transportation system is much more difficult, and it is still a fairly open problem so far.

In the process of transportation, it is difficult to directly control the movement of the lifted-objects, which reflects the underactuated characteristics and brings many challenges. In recent years, the control for various underactuated systems, such as cranes and underwater vehicles, has received considerable interest and achieved satisfactory results by utilizing nonlinear control methods \cite{crane1,crane2,crane3,crane4,crane5,underwater1,underwater2,underwater3}. With the application of advanced sensors, feedback control for the intractable drone has recently attracted extensive attention\cite{drone1,drone2,drone3,drone4}. Some classic control strategies are used to realize various tasks such as aggressive maneuvers, dynamic trajectory generation, and so on. As for the drone transportation system, the load cannot be directly controlled, and the cables exacerbate the system nonlinearities and introduce complex dynamic couplings, these factors make the control problem for the drone transportion system full of all kinds of challenges. So far, researchers have done a lot of work to suppress and eliminate the load sway efficiently, with effective methods including geometric control\cite{geometric2}, dynamic programming\cite{dynnamic1}, differential flatness\cite{geometric1}, time-optimal motion planning\cite{cable3}, hierarchical control\cite{cable2}, and so on. Regardless of the control strategy, the transportation capacity of a single drone is always limited, it is thus urgently needed to implement transportation tasks by multiple drones with high-performance cooperative control strategy.


\textcolor[rgb]{0,0,1}{Compared with the control of a single drone with a suspended load, the study of double drone-bar system presents many challenges in both modeling and control, which are mainly summarized as follows: 1) The bar-load for the double drone-bar transportation system usually has a large size or weight. Thus, the load cannot be regarded as a point; instead, its size and attitude should be considered, making it complicated to describe the swing dynamics accurately.  2)  High degree of freedom implies that more variables are needed to describe the system, including some non-independent variables. Thus, how to deal with the variables without approximation for the modeling and control problem is a great challenge. 3) The proposed method must properly and efficiently coordinate the two drones throughout the transportation process. Swing elimination and cooperative control make the problem even more complicated. Due to the aforementioned facts, fewer research results have been reported for such systems.  }

\textcolor[rgb]{0,0,1}{Due to the aforementioned facts, few research results have been reported for such systems.} A vision-based collaborative control scheme is proposed in \cite{double1}, wherein the motion of the two drones can be well organized even without communication. However, without careful consideration for swing suppression, the speed of the drones has to be strictly limited within a certain range. Some load attitude control method is proposed in \cite{double5}, which drives the bar to the desired pose and guarantees exponential stability. Nevertheless, \cite{double5} ignores the information exchange and does not consider the collaboration between the drones. \textcolor[rgb]{0,0,1}{Linearization is used to infer the equilibrium's exponential stability in \cite{double5}, failing to establish a precise mathematical description for the double drone-bar dynamics. Due to the two drones' independent control, the advantages of information exchange are not fully implemented, limiting the applications in the double drone transportation system. Instead, G. Loianno \emph{et al.} propose a coordinated approach that enables each drone to construct its controller, utilizing not only its own sensor, but also the measurements acquired by other drones \cite{double2}, which inspires this work.} Unfortunately, the inherently rigid structure strictly determines the relative position between the two drones, which badly reduces the flexibility of the transportation system. The formation control law is presented in\cite{multi1} to endow the interconnected system with a continuum of equilibria. \cite{multi2} presents the control system for the transportation of a slung load by one or several helicopters and successfully implements experimental tests. H. Lee \emph{et al.} employ the parameter-robust linear quadratic Gaussian method to solve the problem of slung-load transportation, and prove the stability of the system with Lyapunov analysis \cite{multi3}. It is noted that, since the load is regarded as point mass and there is no actual information exchange between the drones, the methods proposed in \cite{multi1,multi2,multi3} cannot guarantee the smoothness and safety of the transportation process, and they cannot always yield satisfactory experimental results.


\textcolor[rgb]{0,0,1}{In many cases, the load may be too large or too heavy to be transported by only one drone.  Two drones will be used to coordinate the delivery of these goods. With the widespread application, this double drone-bar system's prospects are becoming broader, making the corresponding research very urgent. To solve the above-mentioned practical problems, this paper uses the Lagrangian modeling method to establish an accurate double drone-bar model. Based on the obtained model, a nonlinear coordinated control strategy is proposed to ensure the asymptotic stability of the required balance point. Lyapunov technique and LaSalle's invariance theorem strictly prove the correctness of the problem without turning to any linearization or approximation. Finally, extensive experiments are carried out to verify the effectiveness and robustness of the proposed method.}

The main contributions of this paper are summarized as follows:\textcolor[rgb]{0,0,1}{
\begin{enumerate}
	 \item An accurate model is set up for the double drone-bar transportation system without turning to any linearization or approximation, based on which, a nonlinear cooperative controller is designed, which carefully considers the complex nonlinearity and the strong coupling between the drones, and achieves asymptotic stability as proven by Lyapunov techniques and LaSalle's invariance.
	\item The proposed controller works independently on the two drones, meanwhile, the interconnection of the two drones by load is rigorously analyzed, which helps to effectively suppress the swing and control the relative position of the double drones.
	\item As supported by experimental results, even in the presence of various disturbances, such as wind or manual disturbances, the proposed control method presents strong anti-swing ability and good robustness for different transportation scenarios, superior to existing control methods.
\end{enumerate}}

The remaining part of the paper is organized as follows: Section $2$ uses the Lagrange modeling method to set up the model for the double drone-bar transportation system. Section $3$ divides the transportation system into the inner loop and the outer one, where Lyapunov method is employed to design an advanced control for the outer loop, with LaSalle's invariance principle utilized to analyze its stability. In section $4$, experimental results are included to verify the feasibility and superior performance of the proposed method. Section $5$ provides summaries and conclusions of this paper.

\begin{figure}[t]
	\begin{center}
		\includegraphics[width=7cm]{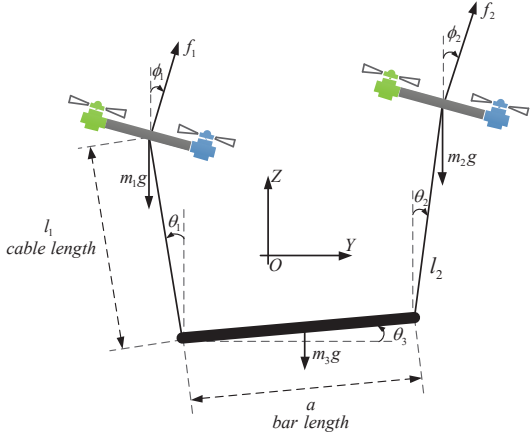}    
		\caption{Modeling of the drone-bar system}  
		\label{fig:2WMR}                                 
	\end{center}                                 
\end{figure}

\section{System Model Development}

\label{sec: systemmodel}
As illustrated in Fig. \ref{fig:2WMR}, there are two drones and a bar load in the transportation system, which are connected by two cables of equal length. For the following part of this paper, we refer to this system as drone-bar system. The drones' and the bar's positions are denoted by $\boldsymbol\xi_{1}(t) = [y_1(t),z_1(t)]^{T} = [y,z]^{T}$, $\boldsymbol\xi_{2}(t) = [y_{2}(t),z_{2}(t)]^{T}=[y+l_{1}S_{1}+l_{2}S_{2}+aC_{3},z-l_{1}C_{1}+l_{2}C_{2}+aS_{3}]^{T}$ and $\boldsymbol\xi_{3}(t) = [y_{3}(t),z_{3}(t)]^{T}=[y+l_{1}S_{1}+\frac{a}{2}C_{3}, z-l_{1}C_{1}+\frac{a}{2}S_{3}]^{T}$, respectively. $m_1,m_2,m_3$ stand for their masses.  $\boldsymbol\Theta(t) = [\theta_{1}(t),\theta_{2}(t),\theta_{3}(t)]^{T}$ denotes the swing signals, and they are defined in the manner as shown in Fig. \ref{fig:2WMR}. $f_{1}\in \mathbb{R}$ and $f_{2}\in \mathbb{R}$ are applied thrust force in the inertial frame, and $J_{1}\in \mathbb{R}^+,J_{2}\in \mathbb{R}^+$ represent the moment of inertia of the two drones. In addition, we denote by $l_{1} = l_{2} =l>0$ the cables' length, by $a>0$ the bar's length, by $g$ the gravity acceleration, by $\phi_1,\phi_2$ the rotating angle of the drones. The model of this drone-bar system is built on a plane perpendicular to the X-axis and constructed according to their geometric relationship.
It is essential in our implementation to decompose control design into two loops as the inner loop and the outer one. \textcolor[rgb]{0,0,1}{The lifting system has seven degrees of freedom, and their corresponding generalized forces in the outer loop subsystem are denoted by $Q_i$: $i \in [1,5],i\in N^+$. Lagrange's method is utilized for system modeling, and after performing a certain amount of mathematical operations, the outer loop dynamics of the double drone-bar system can be derived as
\begin{align}
\label{Q1}
Q_1 &= {f_1}\sin {\phi _1} + {f_2}\sin {\phi _2},\\
\label{Q2}
Q_2 &= {f_1}\cos {\phi _1} + {f_2}\cos {\phi _2},\\
\label{Q3}
Q_3 &= {f_2}\sin {\phi _2}{l_1}{C_1} + {f_2}\cos {\phi _2}{l_1}{S_1},\\
\label{Q4}
Q_4 &= {f_2}\sin {\phi _2}{l_2}{C_2} - {f_2}\cos {\phi _2}{l_2}{S_2},\\
\label{Q5}
Q_5 &= -f_2\sin\phi_2aS_3 +f_2\cos\phi_2aC_3.
\end{align}
}In the transportation task, we need to drive the two drones to the desired positions as $\boldsymbol\xi_{1d} = [y_{1d},z_{1d}]^{T}$, $\boldsymbol\xi_{2d} = [y_{2d},z_{2d}]^{T}$, while suppressing the swing of the load. Also, during the whole transportation process, the distance of the two drones should be limited within a certain range, so that the drones will not collide with each other when they are too close, or break the cables when trying to fly away from each other. In this sense, the mathematical representation of the control target is
\begin{align}
\label{goal}
&y_1\rightarrow y_{1d}, y_2\rightarrow y_{2d}, z_1\rightarrow z_{1d}, z_2\rightarrow z_{2d},\nonumber\\
&\theta_1\rightarrow \theta_{1d}, \theta_2\rightarrow \theta_{2d},
\end{align}
\textcolor[rgb]{0,0,1}{where the desired positions are constant. Generally speaking, the working scenes of multiple drones are mostly used to transport heavier objects, and most of the sling ropes used in the experiment are steel ropes that are not easy to bend.} Similar with the work presented in \cite{cable3,cable2,nest,na,doup3,und2}, we make the following reasonable assumption about the cables and the swing angles:\textcolor[rgb]{0,0,1}{
\begin{assumption}
	\label{assumption1}
	Aggressive motion control is not taken into account. The two cables in the transportation system are always in tension. Besides, the bar is always under the two drones in the vertical direction, and the bar is not in a straight line with either of the two cables, i.e.
	\begin{align}
	-\frac{\pi}{2}<\phi_{1}, \phi_{2},\theta_1(t),\theta_2(t),\theta_3(t)<\frac{\pi}{2}\quad\quad\forall \,t\ge 0.\nonumber
	\end{align}
\end{assumption}}

\section{Control Development}
\label{sec: Controller}

In this section, a nonlinear hierarchical control scheme is used to facilitate the design procedure. As \cite{cable1} shows, drone transportation systems have the cascade property, thus controllers can be designed for the inner loop and the outer one separately. The contribution of this paper mainly focuses on the design of the outer loop controller, and the inner loop adopts the controller presented in \cite{cable1,geometric2,geometric1}, which performs well in experimental tests.

\subsection{Outer Loop Controller Design}
\label{sec: outer loop control}

First, define the following error signals:
\begin{align}\label{error signals}
&\boldsymbol{e}_{{\xi}_{1}} = [e_{y_{1}},e_{z_{1}}]^T=[ y_1-y_{1d},z_1-z_{1d}]^T,\nonumber\\
&\boldsymbol{e}_{{\xi}_{2}} = [e_{y_{2}},e_{z_{2}}]^T=[ y_2-y_{2d},z_2-z_{2d}]^T.
\end{align}
To facilitate subsequent controller development and analysis, define the desired signals to satisfy:
\begin{align}
\label{relation0}
&y_{2d}-y_{1d}= a, z_{1d}=z_{2d},\theta_{1d}=\theta_{2d}=0.
\end{align}
Considering the underactuation property of the system, energy-based method is adopted to design the cooperative control. To this end, calculate the storage energy of the drone-bar system as:
\begin{align}
\label{energy}
E\! = \! \frac{1}{2}{\dot {\boldsymbol q}^T}\! M \!(\boldsymbol{q})\dot {\boldsymbol q} \!+ \! m_{3}g \! \left[  \frac{1}{2}{l_1}(1 \! -\! C_1)\! +\! \frac{1}{2}{l_2} ( 1 \!- \! C_2)\right],
\end{align}
\textcolor[rgb]{0,0,1}{where the state vector $\boldsymbol{q}\in\mathbb{R}^5$ and the inertia matrix $M(\boldsymbol{q})\in \mathbb{R}^{5\times 5}$ in (\ref{energy}) are defined as
\begin{align}
\boldsymbol{q}(t) &= \left[y(t),z(t),\theta_{1}(t),\theta_{2}(t),\theta_{3}(t)\right]^T,\\
M(\boldsymbol q) &= \left[ {\begin{array}{*{20}{c}}
	{{M_{11}}}&0&{{M_{13}}}&{{M_{14}}}&{{M_{15}}}\\
	0&{{M_{22}}}&{{M_{23}}}&{{M_{24}}}&{{M_{25}}}\\
	{{M_{31}}}&{{M_{32}}}&{{M_{33}}}&{{M_{34}}}&{{M_{35}}}\\
	{{M_{41}}}&{{M_{42}}}&{{M_{43}}}&{{M_{44}}}&{{M_{45}}}\\
	{{M_{51}}}&{{M_{52}}}&{{M_{53}}}&{{M_{54}}}&{{M_{55}}}
	\end{array}} \right],
\end{align}
with
\begin{align}
&{M_{11}}\! = \!{M_{22}} \!= \!{m_1}\! +\! {m_2}\! + \!{m_3}\!,\!{M_{55}}\! = \!{m_2}{a^2}\! + \!{m_3}{\left( {\frac{a}{2}} \right)^2},\nonumber\\
&{M_{13}} = {M_{31}} = {m_2}{l_1}{C_1} + {m_3}{l_1}{C_1},\nonumber\\
&{M_{14}} = {M_{41}} = {m_2}{l_2}{C_2},\,{M_{24}} = {M_{42}} =  - {m_2}{l_2}{S_2},\nonumber\\
&{M_{15}} = {M_{51}} =  - {m_2}a{S_3} - {m_3}\frac{a}{2}{S_3},\nonumber\\
&{M_{23}}\!\! = \!\!{M_{32}} \!\!=\! \!{m_2}{l_1}{S_1}\!\! + \!\!{m_3}{l_1}{S_1}\!,\!{M_{25}} \!\!=\! \!{M_{52}}\! \!=\! \!{m_2}a{C_3}\! \!+\!\! {m_3}\frac{a}{2}{C_3},\nonumber\\
&{M_{33}} = {m_2}{l_1}^2 + {m_3}{l_1}^2,\,{M_{44}} = {m_2}{l_2}^2,\nonumber\\
&{M_{34}}\! =\! {M_{43}}\! = \!{m_2}{l_1}{l_2}{C_{1 + 2}}\!,\!{M_{45}} = {M_{54}}\! =\!  - {m_2}{l_2}a{S_{2\! + \!3}},\nonumber\\
&{M_{35}} = {M_{53}} = {m_2}{l_1}a{S_{1 - 3}} + {m_3}{l_1}\frac{a}{2}{S_{1 - 3}}.\nonumber\\
\end{align}}
Taking the time derivative of $E$ results in
\begin{align}
\label{denergy}
\dot E\! = &{\dot y_1}f_1\sin\phi _1 \!+\! {\dot y_2}f_2\sin\phi _2\!+ \!{\dot z_1}\!\left[ {{f_1}\cos {\phi _1}
	\!-\! \big(\! {{m_1}\! + \!\frac{1}{2}{m_3}} \!\big)g} \right]\nonumber\\
&+ {\dot z_2}\left[ {{f_2}\cos {\phi _2} - \big( {{m_2} + \frac{1}{2}{m_3}} \big)g} \right].
\end{align}
Decomposing the applied thrust $f_1$, $f_2$ into $f_1\sin\phi_1$, $f_1\cos\phi_1$, $f_2\sin\phi_2$, $f_2\cos\phi_2$ as shown in (\ref{denergy}), the control inputs are construsted as follows:
\begin{align}
\label{input1}
{f_1}\sin{\phi _1} =&  - k_{p_1}{e_{{y_1}}} - {k_{d_1}}{\dot y_1} - k_{a_1}\left( {{{\dot \theta }_1}^2 + {{\dot \theta }_2}^2 + {{\dot \theta }_3}^2} \right){\dot y_1}\nonumber \\
&- \frac{1}{2}{m_3}g\tan {\theta _{2d}} - \frac{\sigma\!\rho e_{y}}{\left( \rho \!-\!e_y^2\right) ^2 }\\
\label{input2}
{f_2}\sin{\phi _2} =&  - k_{p_2}{e_{{y_2}}} - {k_{d_2}}{\dot y_2}- k_{a_2}\left( {{{\dot \theta }_1}^2 + {{\dot \theta }_2}^2 + {{\dot \theta }_3}^2} \right){\dot y_2}\nonumber \\
&+ \frac{1}{2}{m_3}g\tan {\theta _{2d}} + \frac{\sigma\!\rho e_{y}}{\left( \rho \!-\!e_y^2\right) ^2 }\\
\label{input3}
{f_1}\cos{\phi _1} =&  - k_{p_3}{e_{{z_1}}} - {k_{d_3}}{\dot z_1} + \frac{1}{2}\left( {2{m_1} + {m_3}} \right)g,\\
\label{input4}
{f_2}\cos {\phi _2} =&  - k_{p_4}{e_{{z_2}}} - {k_{d_4}}{\dot z_2} + \frac{1}{2}\left( {2{m_2} + {m_3}} \right)g,
\end{align}
where $k_{p_1}$, $k_{p_2}$, $k_{p_3}$, $k_{p_4}$, $k_{d_1}$, $k_{d_2}$, $k_{d_3}$, $k_{d_4}$, $k_{a_1}$, $k_{a_2}$, $\sigma$ are positive control gains. In addition, $e_y = e_{y_{1}}-e_{y_{2}}$ and $\rho$ is a positive constant satisfying
\begin{align}
\label{rho}
\rho > \left| e_{y_1}(0) - e_{y_2}(0)\right| ^2,
\end{align}
with $e_{y_1}(0),e_{y_2}(0)$ being initial values. \textcolor[rgb]{0,0,1}{It is worth noting that solving the problem of unknown drone and bar-load mass will overturn most of the previous conclusions, making the difficult control problem more complicated. Thus, we start with the {\emph{Exact Model Knowledge} controller design and theoretical analysis.}}

{\bf\emph{Remark 1:}} The third term in the proposed control input (\ref{input1})-(\ref{input2}) is to enhance the coupling of the system. In general, the traditional model-free PD controller cannot respond to load swing in a timely and efficient manner, making it difficult to achieve satisfactory control effects. The designed term will change accordingly to react efficiently to the swing when the bar is oscillating, which enhances the coupling of the system and facilitates the suppression of the swing.

{\bf\emph{Remark 2:}} The last term in the proposed control inputs (\ref{input1})-(\ref{input2}) is to ensure the distance between the two drones within the desired range, and the subsequent experiments can demonstrate such a capability. For instance, under the influence of external disturbances, the positions of the two drones may be too close, implying that the distance between the two drones in the horizontal direction is about to reach the boundary of the desired range, thus $e_{y}^2\rightarrow \rho^-$, directly making $\frac{\sigma\!\rho e_{y}}{\left( \rho \!-\!e_y^2\right) ^2 } \rightarrow \infty$. At this point, control input (\ref{input1})-(\ref{input2}) will drive the drones away from each other to avoid collision. Similar conclusions can be drawn in the other cases. In this way, owing to the introduction of the last term, the two drones will not collide or go too far away from each other. In practical applications, we can set a proper value for $\rho$ according to the demand of different transportation process to achieve the desired control performance.

\textcolor[rgb]{0,0,1}{{\bf\emph{Remark 3:}} We try to keep the desired position of the bar-load always within the Y-Z plane during the flight. The error signal in the X-axis may come from wind disturbances or other factors. Thus, we design a traditional PID controller to realize the positioning of drones in the X-axis. As supported by subsequent experimental results, the drones' positioning error in the X-direction is usually $\pm 1\,\mathrm{cm}$, and it is  $\pm 3 \, \mathrm{cm}$ even under wind disturbances.}

\textcolor[rgb]{0,0,1}{{\bf\emph{Remark 4:}} By invoking the theory on cascade systems, controllers can be designed for the inner loop and the outer one separately. This manuscript is concentrated on the outer loop subsystem, thus, the controller design procedure and stability analysis are detailed. The inner loop dynamics on the rotation motion of the quadrotor stay the same as those given in literatures \cite{lee2010geometric}}.
	
\textcolor[rgb]{0,0,1}{{\bf\emph{Remark 5:}} By redefining the energy storage energy function given in (9) as follows, the designed controller can be extended to a tracking controller:
	\begin{align}
		E\! =& \frac{1}{2}{\dot {\boldsymbol e}^T}M(\boldsymbol{q})\dot {\boldsymbol e}\!+ \!m_{3}g\left[  \frac{1}{2}{l_1}(1 - C_1)\! +\! \frac{1}{2}{l_2}(1 - C_2)\right], \nonumber
	\end{align}
where the tracking error vector is defined as ${\boldsymbol e}(t)=[e_{y1}(t),e_{z1}(t),e_{y2}(t),e_{z2}(t), \theta_{1}(t), \theta_{2}(t) ]$, where $e_{y1}(t) = y_1(t)-y_{1d}(t), \, e_{z1}(t) = z_1(t)-z_{1d}(t),\,e_{y2}(t) = y_2(t)-y_{2d}(t), \, e_{z2}(t) = z_2(t)-z_{2d}(t)$.Then, with the similar process in this manuscript, the tracking controller can be designed. It can be proved that the system is asymptotically stable by using Lyapunov technique and Barbalat's lemma.}

\subsection{Stability Analysis}
\label{sec:stability}

\begin{theorem}
	For the nonlinear underactuated double drone-bar system, the designed controller given by (\ref{input1})-(\ref{input4}) guarantees that the desired equilibrium point depicted in (\ref{goal}), is asymptotically stable, in the sense of
	
	\begin{align}
		\lim_{t \to \infty} \boldsymbol e_{\xi_{1}} = \boldsymbol 0,\lim_{t \to \infty} \boldsymbol e_{\xi_{2}} = \boldsymbol 0, \lim_{t \to \infty} {\boldsymbol\Theta} = \boldsymbol 0
	\end{align}

\end{theorem}
\begin{proof}
	First of all, based on the analysis for the storage energy of the system, the following Lyapunov candidate function is chosen:
	\begin{align}
	\label{lyaounov}
	V(t)\! =& \frac{1}{2}{\dot {\boldsymbol q}^T}\!M\!(\boldsymbol{q})\dot {\boldsymbol q}\!\! +\! \frac{1}{2}k_{p_1}e_{{y_1}}^2 \!\!+\! \frac{1}{2}k_{p_2}e_{{y_2}}^2\! \!+\! \frac{1}{2}k_{p_3}e_{{z_1}}^2 \!\!+\! \frac{1}{2}k_{p_4}e_{{z_2}}^2\nonumber\\
	&\!+ \!m_{3}g\left[  \frac{1}{2}{l_1}(1 - C_1)\! +\! \frac{1}{2}{l_2}(1 - C_2)\right]\!+\!\frac{\sigma e_{y}^2}{2\left( \rho \!-\!e_y^2\right)  } \nonumber\\
	& \!+ \!\frac{1}{2}{m_3}ge_y\tan{\theta _{2d}},
	\end{align}
	Taking the time derivative of $V(t)$, and substituting (\ref{input1})-(\ref{input4}) into $\dot{V}(t)$ yields
	\begin{align}
	\label{dotlya}
	\dot V(t) =&-k_{a_1}\left( {{{\dot \theta }_1}^2 + {{\dot \theta }_2}^2 + {{\dot \theta }_3}^2} \right){{\dot y}_1}^2 - k_{a_2}\left( {{{\dot \theta }_1}^2 + {{\dot \theta }_2}^2 + {{\dot \theta }_3}^2} \right){{\dot y_2}^2}\nonumber\\
	& - {k_{d_1}}{\dot y_1}^2 - {k_{d_2}}{\dot y_2}^2 - {k_{d_3}}{\dot z_1}^2 - {k_{d_4}}{\dot z_2}^2.
	\end{align}
	After some mathematical calculation, the following result can be concluded from (\ref{dotlya}):\textcolor[rgb]{0,0,1}{
	\begin{align}
	\label{limit}
	\dot V(t) \le 0 \Rightarrow {V\left( {{t}} \right)} \le {V\left( 0 \right)}<  + \infty ,
	\end{align}
}where the fact of $\rho > \left| e_{y_1}(0) - e_{y_2}(0)\right| ^2$ shown in (\ref{rho}) is utilized. More specifically, suppose $\left| e_{y_1}(t) - e_{y_2}(t)\right|^2$ tends to exceed the boundary of $\rho$ from interior, then there exists some time making $\left| e_{y_1}(t) - e_{y_2}(t)\right|^2 \!\rightarrow\!\rho^- \Rightarrow \frac{\sigma e_{y}^2}{2\left( \rho \!-\!e_y^2\right)  }\!\rightarrow\! +\infty \Rightarrow V(t) \!\rightarrow\! +\infty$, which conflicts with the conclusion in (\ref{limit}). Therefore, it is concluded that $\frac{\sigma e_{y}^2}{2\left( \rho \!-\!e_y^2\right)}\ge 0$, which further indicates $\frac{1}{2}{m_3}ge_y\tan{\theta _{2d}}\ge-\frac{1}{2}{m_3}g\sqrt{\rho}|\tan{\theta _{2d}}|$. $M(\boldsymbol{q})$ in (\ref{lyaounov}) is a positive definite matrix, thus $(\frac{1}{2}{\dot q^{{T}}}M\dot q)\ge 0$. $k_{p_{1}},k_{p_{2}},k_{p_{3}},k_{p_{4}}$ are positive gains, hence, the term $(\frac{1}{2}k_{p_{1}}e_{y_{1}}^2 + \frac{1}{2}k_{p_{2}}e_{y_{2}}^2 + \frac{1}{2}k_{p_{3}}e_{z_{1}}^2 + \frac{1}{2}k_{p_{4}}e_{z_{2}}^2)\ge 0$. Besides, it is obvious that $m_{3}{g}[ \frac{1}{2}{l_1}({1 - C_1})+\frac{1}{2}{l_2}(1 - {C_2})] \ge m_{3}g[ {\frac{1}{2}{l_1}\left( {1 - 1} \right) + \frac{1}{2}{l_2}\left( {1 - 1} \right)}]$. Consequently, because the terms contained in $V(t)$ are all lower bounded, the function $V(t)$ itself is also lower bounded. Utilizing this fact, together with the conclusion in (\ref{limit}), one can obtain that
	\begin{align}
	V \in {\mathcal{L}_\infty } \Rightarrow &{{{\dot y}}_{{1}}},{\dot z_{{1}}},{\dot \theta _{{1}}},{\dot \theta _2},{\dot \theta _3},{e_{{y_{{1}}}}},{e_{{y_2}}},{y_{{1}}},{z_1} \in {\mathcal{L}_\infty }\nonumber\\
	&\frac{\sigma e_{y}^2}{2\left( \rho \!-\!e_y^2\right)  }\in {\mathcal{L}_\infty }.
	\end{align}
	To prove the asymptotic stability of the designed outer loop controller, the set $\boldsymbol\Omega  = \left\{ {\left( {\boldsymbol q,\dot {\boldsymbol q}} \right)\left| {\dot V = 0} \right.} \right\}$ is defined, and $\boldsymbol\Psi$ is the largest invariant set in $\boldsymbol\Omega$. For clarity, the proof process is divided into the following two steps.
	
	{\bf{Step 1}}: In this part, it will be shown that $e_{z_{1}}=e_{z_{2}}=0$. According to equation (\ref{dotlya}), one can obtain that
	\begin{align}
	\label{dotv=0}
	\dot V\! =\! 0 \Rightarrow {\dot y_1} \!=\! {\dot y_2} \!=\! {\dot z_1} \!= \!{\dot z_2} \!= \!0,{\ddot y_1} \!=\! {\ddot y_2} \!= \!{\ddot z_1} \!=\! {\ddot z_2} \!= \!0.
	\end{align}
	According to the expressions in (\ref{Q1}), adding (\ref{input1}) and (\ref{input2}) yields
	\begin{align}
	\label{key-Q1}
	- k_{p_1}{e_{{y_1}}}	& - k_{p_2}{e_{{y_2}}} = {f_1}\sin {\phi _1} + {f_2}\sin {\phi _2}.
	\end{align}
	Substituting $\ddot y(t)=0$ into (\ref{Q1}) and integrating both sides of (\ref{key-Q1}) with respect to time, one can obtain that
	\begin{align}
	\label{key-integrate}
	& \left( {{m_2} \!+ \!{m_3}} \right){l_1}{C_1}{\dot \theta _1} \!+ \!{m_2}{l_2}{C_2}{\dot \theta _2}
	\!-\!\left( { {m_2} \!+\!\frac{1}{2}{m_3}} \right)a{S_3}{\dot \theta _3} \nonumber\\
	=&\left( { - k_{p_1}{e_{{y_1}}} \!-\! k_{p_2}{e_{{y_2}}}} \right)t\! + \!{\beta _1},
	\end{align}
	wherein $\beta_1$ denotes a constant to be determined. If $ (- k_{p_1}{e_{{y_1}}} - k_{p_2}{e_{{y_2}}})\ne 0$, when $t\rightarrow \infty$, both sides of (\ref{key-integrate}) tend to infinity, which has obvious contradiction with the inference ${\dot \theta _1},{\dot \theta _2},{\dot \theta _3} \in {\mathcal{L}_\infty }$ in (\ref{dotv=0}). Then the following conclusions are drawn:
	\begin{align}
	\label{key-inte2}
	- k_{p_1}{e_{{y_1}}} - k_{p_2}{e_{{y_2}}}= 0\Rightarrow{e_{{y_{1}}}} =  - \frac{{k_{p_2}}}{{k_{p_1}}}{e_{{y_2}}}.
	\end{align}
	\textcolor[rgb]{0,0,1}{ Substituting (\ref{key-inte2}) into (\ref{key-integrate}) and integrating both sides of equation (\ref{key-inte2}) with time yields
	\begin{align}
	\label{key-inte3}
	\left(\!  {{m_2} \! \!+\!\!  {m_3}}\!  \right){l_1}{S_1} \! \!+\!  \!{m_2}{l_2}{S_2} \!\! +\! \! a{C_3}\! \left( \! {{m_2} \!+\! \frac{m_{3}}{2}}\!  \right)\! =\! {\beta _1}t \! \! +\! \! {\beta _2},
	\end{align}
	wherein $\beta_2$ denotes a constant to be determined. If $\beta_1\ne 0$, with similar analysis for (\ref{key-integrate})-(\ref{key-inte2}), one can draw the following conclusion:
	\begin{align}
	\label{key-inte4}
	\left( \! {{m_2}\! \! +\! \! {m_3}} \! \right){l_1}{C_1}{\dot \theta _1} \!\! +\! \! {m_2}{l_2}{C_2}{\dot \theta _2} \! \!-\! \! a{S_3}{\dot \theta _3}\left( \!   m_2 \!+ \!\frac{m_3}{2}\!  \right) \!=\! 0.
	\end{align}
	Taking the time derivative of $\boldsymbol{\xi}_{2}(t)$, we obtain the velocity $\dot{\boldsymbol{\xi}}_{2}(t)$ of the drone on the right. Combining with the equation (\ref{key-inte4}), the following equations can be obtained:
	\begin{align}
	\label{key-case1}
	\begin{cases}
	\! \left( \! {{m_2} \!\! +\! \! {m_3}} \! \right)\! {l_1}\! {C_1}{\dot \theta _1}\!  \!+\!\!  {m_2}\! {l_2}\! {C_2}{\dot \theta _2}\! \! -\! \! a{S_3} {\dot \theta _3}\! \! \left(  \!  m_2\! \! +\! \! \frac{m_3}{2}\!  \right)\! \! =\!0,  \\
	{l_1}{C_1}{{\dot \theta }_1} + {l_2}{C_2}{{\dot \theta }_2} - a{S_3}{{\dot \theta }_3} = 0,\\
	{l_1}{S_1}{{\dot \theta }_1} - {l_2}{S_2}{{\dot \theta }_2} + a{C_3}{{\dot \theta }_3} = 0.
	\end{cases}
	\end{align}
	{\bf{Lemma 1:}} $\dot\theta_{1}\!=\!\dot\theta_{2}\!=\!\dot\theta_{3}\!=\!0$ is the only solution for (\ref{key-case1}).\\
	{\bf{Proof:}} please see the appendix for details.}
	
	According to Cramer's Rule, one can obtain that
	\begin{align}
	\label{dot-theta=0}
	{\dot \theta _1}, {\dot \theta _2}, {\dot \theta _3} = 0 \Rightarrow {\ddot \theta _1},{\ddot \theta _2},{\ddot \theta _3} = 0.
	\end{align}
	According to (\ref{dot-theta=0}) and the expressions in (\ref{Q2}), adding (\ref{input3}) and (\ref{input4}) yields
	\begin{align}
	\label{relation1}
	&{f_1}\cos {\phi _1} + {f_2}\cos {\phi _2} = \left( {{m_{{1}}}{{ + }}{m_{{2}}}{{ + }}{m_{{3}}}} \right){{g}} \nonumber\\
	\Rightarrow & - k_{p_3}{e_{{z_1}}} - k_{p_4}{e_{{z_2}}} = 0.
	\end{align}
	It is worth noting from (\ref{dotv=0}) that when $\dot V=0,\ddot z_{1}=\ddot z_{2}=0$, both of the drones do not provide any acceleration in the vertical direction, thus, they only bear the load and their own gravity. According to Assumption1, we can set the following restrictions on $f_{1}\cos\phi_1,f_2\cos\phi_{2}$ :
	\begin{align}
	\label{relation2}
	{m_1}g < {f_1}\cos {\phi _1} < \left( {{m_1}{ + }{m_3}} \right)g,\nonumber\\
	{m_2}g < {f_2}\cos {\phi _2} < \left( {{m_2}{ + }{m_3}} \right)g.
	\end{align}
	Substituting the results of (\ref{dot-theta=0}) into (\ref{Q3})-(\ref{Q5}), the following result can be obtained :
	\begin{align}
	\label{case2}
	\begin{cases}
	\left( {{m_2} + {m_3}} \right)g{S_1} = {f_2}\sin {\phi _2}{C_1} + {f_2}\cos {\phi _2}{S_1},\\
	- {m_2}g{S_2} = {f_2}\sin {\phi _2}{C_2} - {f_2}\cos {\phi _2}{S_2},\\
	{m_2}g{C_3}\! \!  +\! \!  \frac{1}{2}\! {m_3}g{C_3} \! = \!  -\!  {f_2}\! \sin {\phi _2}\! {S_3} \! \! +\!  \! {f_2}\! \cos {\phi _2}{C_3}.
	\end{cases}
	\end{align}
	Substituting (\ref{input4}) into (\ref{case2}), one can obtain that
	\begin{align}
	\label{case3}
	\begin{cases}
	\tan {\theta _1}\left( {\frac{1}{2}{m_3}g + k_{p_4}{e_{{z_2}}}} \right) = {f_2}\sin {\phi _2},\\
	\tan {\theta _2}\left( {\frac{1}{2}{m_3}g - k_{p_4}{e_{{z_2}}}} \right) = {f_2}\sin {\phi _2},\\
	- k_{p_4}{e_{{z_2}}} = {f_2}\sin {\phi _2}\tan {\theta _3}.
	\end{cases}
	\end{align}
	{\bf{Lemma 2}}: The only solution for (\ref{case3}) is ${e_{{z_2}}} = 0$.\\
	{\bf{Proof}}: Please see appendix for details.
	
	It then follows from Lemma 1 that
	\begin{align}
	\label{relation3}
	{e_{{z_2}}} = 0,{\theta _1} = {\theta _2}{{, }}{\theta _3} = {{0}} \Rightarrow {e_{{z_1}}} = 0.
	\end{align}
	
	{\bf{Step 2}}: In this part, it will be further proved that	$\theta_{1} = \theta_{1d}, \theta_{2} = \theta_{2d}, e_{y_{1}} = e_{y_{2}} = 0$. Firstly, substituting the results of (\ref{input2}) into the second equation of (\ref{case3}), and combining the conclusions in (\ref{dotv=0}) and (\ref{relation3}), the following result can be obtained after simplification:
	\begin{align}
	\label{new1}
	- k_{p_2}{e_{{y_2}}} +\frac{\sigma\!\rho e_{y}}{\left( \rho \!-\!e_y^2\right) ^2 } = \frac{1}{2}{m_3}g(\tan {\theta _2}-{\tan}\theta_{2d}).
	\end{align}
	\textcolor[rgb]{0,0,1}{ Regarding the right half part of (\ref{new1}), the following inference can be made based on (\ref{relation0}) and (\ref{relation3}):
	\begin{align}
	\label{right}
	{\mathrm{sgn}} \! \left( \! {\tan {\theta _2\! \! -\! \! \tan\theta_{2d}}} \! \right)
    \! = \! {\mathrm{sgn}} \Big(\!  {\big( \! {1 \! \! +\! \!  \frac{{k{p_2}}}{{k{p_1}}}} \big){e_{{y_2}}}} \! \Big) \! = \! {\mathrm{sgn}} \! \left(\!  {{e_{{y_2}}}} \! \right).
	\end{align}
}For the left side of (\ref{new1}), based on (\ref{key-inte2}), the following inference can be made:
	\begin{align}
	&- k_{p_2}{e_{{y_2}}} \!+\! \frac{\sigma\!\rho e_{y}}{\left( \rho \!-\!e_y^2\right)^2 }\! =\! -{e_{{y_2}}}\bigg(k_{p_2}\! +\!\frac{\sigma\!\rho (1+\frac{k_{p_{2}}}{k_{p_{1}}})}{\left[ \rho \!-\!(e_{y_{1}}-e_{y_{2}})^2\right] ^2 }  \bigg) \nonumber\\
	\Rightarrow&\,\,\,\, {\mathrm{sgn}} \bigg(- k_{p_2}{e_{{y_2}}} + \frac{\sigma\!\rho e_{y}}{\left( \rho \!-\!e_y^2\right)^2 } \bigg)-{\mathrm{sgn}} \left( {{e_{{y_2}}}} \right).
	\end{align}
	Therefore, we can conclude that the two sides of the equation (\ref{new1}) are opposite, making ${e_{{y_1}}} = 0$ as its only solution. Some similar analysis can be implemented for the signal ${e_{{y_2}}}$, and the following conclusions can be then drawn:
	\begin{align}
	\label{new2}
	{e_{{y_1}}}= {e_{{y_2}}}= 0.
	\end{align}
	
	Combining the results presented in (\ref{dotv=0}), (\ref{relation1}), (\ref{relation3}), (\ref{new2}), there is a unique solution of $\dot V=0$. Therefore, we can obtain that
	\begin{align}
	\label{conclution1}
	\begin{cases}
	{y_1} = {y_{1d}},{y_2} = {y_{2d}},{z_1} = {z_2} = {z_{1d}} = {z_{2d}},\\
	{\theta _1} = {\theta _2} = {\theta _{1d}} = {\theta _{2d}},{\theta _3} = 0.
	\end{cases}
	\end{align}
	
	This calculation implies that $V=0$ is the unique minimum value at the desired equilibrium point. Based on this conclusion, the results of (\ref{dotlya})-(\ref{new2}) are utilized again to conclude that the closed-loop system is stable, and the maximum invariant set $\boldsymbol\Psi$ only contains the desired equilibrium point. According to LaSalle's invariance theorem, it is concluded that the equilibrium point is asymptotically stable.
\end{proof}

\section{Hardware Experiments}
In this section, two groups hardware experimental results are provided to verify the efficiency of the proposed method, especially its robustness against uncertain parameters and external disturbances. The relevant experimental video has been taken and uploaded to YouTube, specifically posted at (\url{https://youtu.be/hHg66bE2PRA}).

\begin{figure}[!htp]
	\centering
	\includegraphics[width=3.2in]{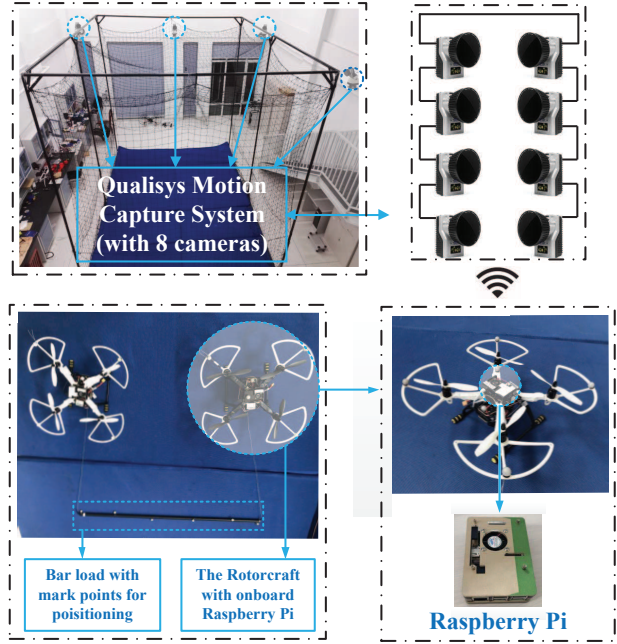}
	\caption{Experiment testbed.}
	\label{fig:testbed}
\end{figure}
The self-built hardware experimental testbed for the drone-bar system is shown in Fig. \ref{fig:testbed}.
The poses of the two drones and the bar load are obtained through the Qualisys Motion Capture System. Data of 8 cameras are sent to the ground station via the Robot Operating System (ROS) based transmission protocol under the LAN. The corresponding control input and torque are calculated on the ground station and then sent to the onboard computer via the WIFI with the 5G band. The flight controller is PIXHAWK, which is connected between the onboard computer and PIXHAWK by mavros based communication protocol. The Raspberry Pi runs the 64-bit Ubuntu-mate 16.04 operating system. The utilization of 5G band wireless network makes its network latency approximately 1-3 milliseconds, which is the best data transmission scheme after a lot of debugging tests. The physical parameters of the experimental testbed are given as follows:
\begin{align}\label{physical parameters}
&m_{1} = m_{2} = 1.5\,{\rm kg},\,m_{3} = 0.3\,{\rm kg},\, g = 9.8\,{\rm m/s^2}.\nonumber\\
& a = 1.2\,{\rm m}, l_{1}=l_{2} = 0.9\,{\rm m}.
\end{align}
 Then, based on the experimental platform, two groups of experiments is carried out to verify the effectiveness of the proposed method.

\subsection{Experiment Group 1}

\begin{table*}[t]
	\renewcommand\arraystretch{1.8}
	\centering
	\caption{Parameters for Experiment 1}
	\label{tab:data_exp1}
	\begin{tabular}{ccc}
		\hline \hline
		{\bf{Exp 1}}    & {\bf{Initial positions}}${\rm (m)}$ & {\bf{Desired positions}}${\rm (m)}$  \\
		\hline
		{Test 1} &  $\quad\quad\quad[y_{1},z_{1}] =[-0.1,1.3],\,[y_{2},z_{2}] = [\,\,\,\,1.5,1.3]\quad\quad\quad$ &  $[y_{1d},z_{1d}] = [-1.3,1.8],\,[y_{2d},z_{2d}] = [-0.1,1.8]$\\
		\hline
		{Test 2}    &  $[y_{1},z_{1}] = [-1.3,1.5],\,[y_{2},z_{2}] =[-0.2,1.5]$ &  $[y_{1d},z_{1d}] = [\,\,\,\,0.0,1.9],\,[y_{2d},z_{2d}] = [\,\,\,\,1.5,1.9]$\\   \hline
		\hline
	\end{tabular}
\end{table*}

In this group of experiments, the proposed controller (\ref{input1})-(\ref{input4}) allows the two drones to reach the desired positions and then hover at that location. To achieve satisfactory convergence speed and anti-swing objectives, we repeat several times for each experiment to adjust the parameters based on experience. The control gains are listed as: 
\begin{align}\label{adjusted parameters}
&k_{p_{1}} = k_{p_{2}} = 5.2, k_{d_{1}}=k_{d_{2}} = 6.0,\nonumber\\
&k_{p_{3}} = k_{p_{4}} = 6.0,  k_{d_{3}}=k_{d_{4}} = 8.0,\nonumber\\
&\delta =4.0 ,\rho =2.0 ,k_{a_{1}}=k_{a_{2}} = 0.75.
\end{align}

Also, we implement some comparative experiments with the classic PD control as the comparative controller, whose control gains are set as $k_{p_{1}} = k_{p_{2}} = 5.2, k_{d_{1}}=k_{d_{2}} = 6.0,k_{p_{3}} = k_{p_{4}} = 6.0,  k_{d_{3}}=k_{d_{4}} = 8.0$ after sufficient tuning. To further verify the robustness of the proposed control method against system uncertainties and various disturbances, the following three tests are carried out.	
\begin{figure}[!htb]
	\centering
	\clearcaptionsetup{figure}
	\clearcaptionsetup{subfloat}
	\captionsetup[subfloat]{labelsep=period}
	\subfloat[ Positions of the two drones for Exp 1-Test 1.]{\includegraphics[width=0.46\textwidth]{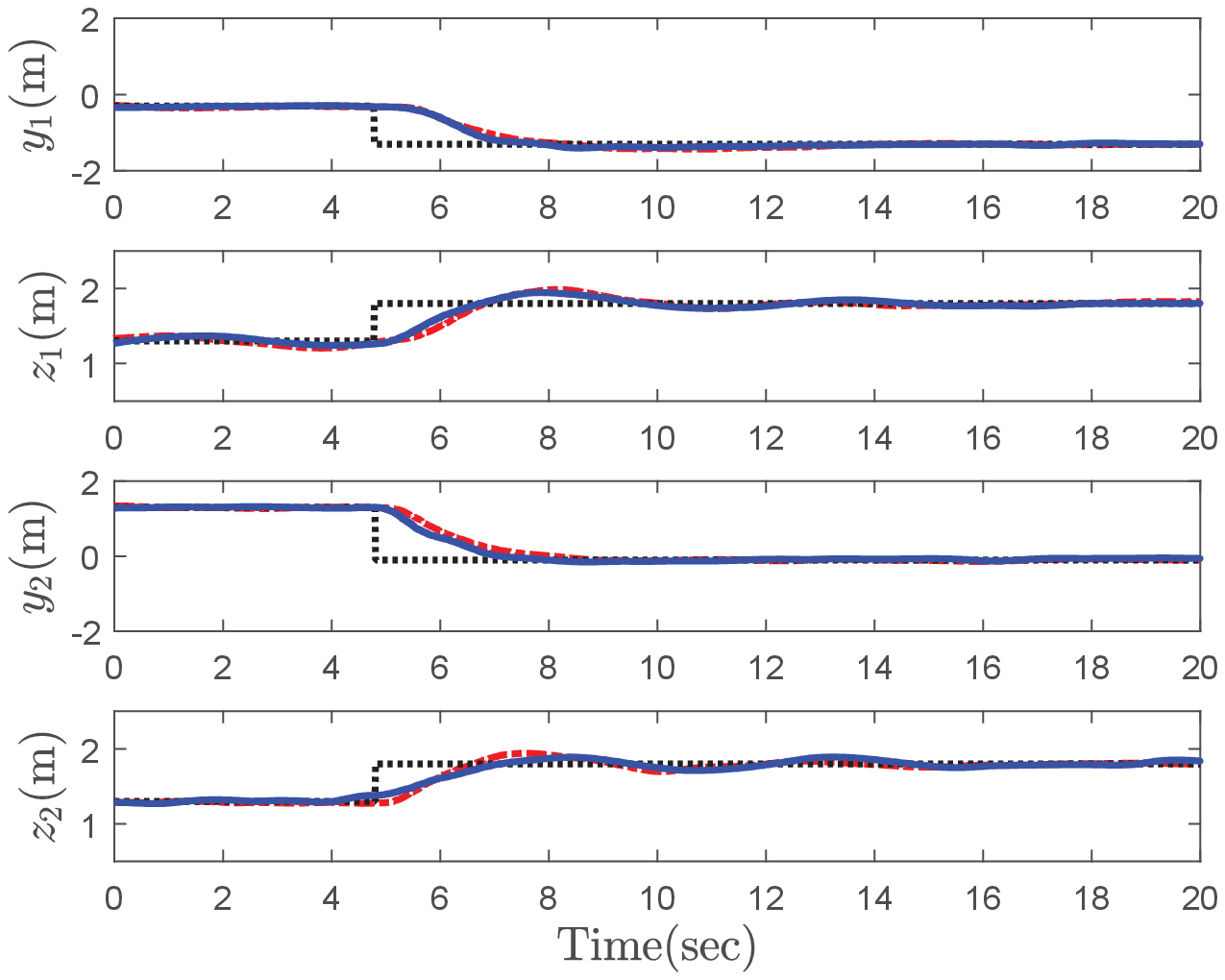}}\hspace{0.0mm}
	\captionsetup[subfloat]{labelsep=period}
	\subfloat[ Control input \& Swing angles.]{\includegraphics[width=0.46\textwidth]{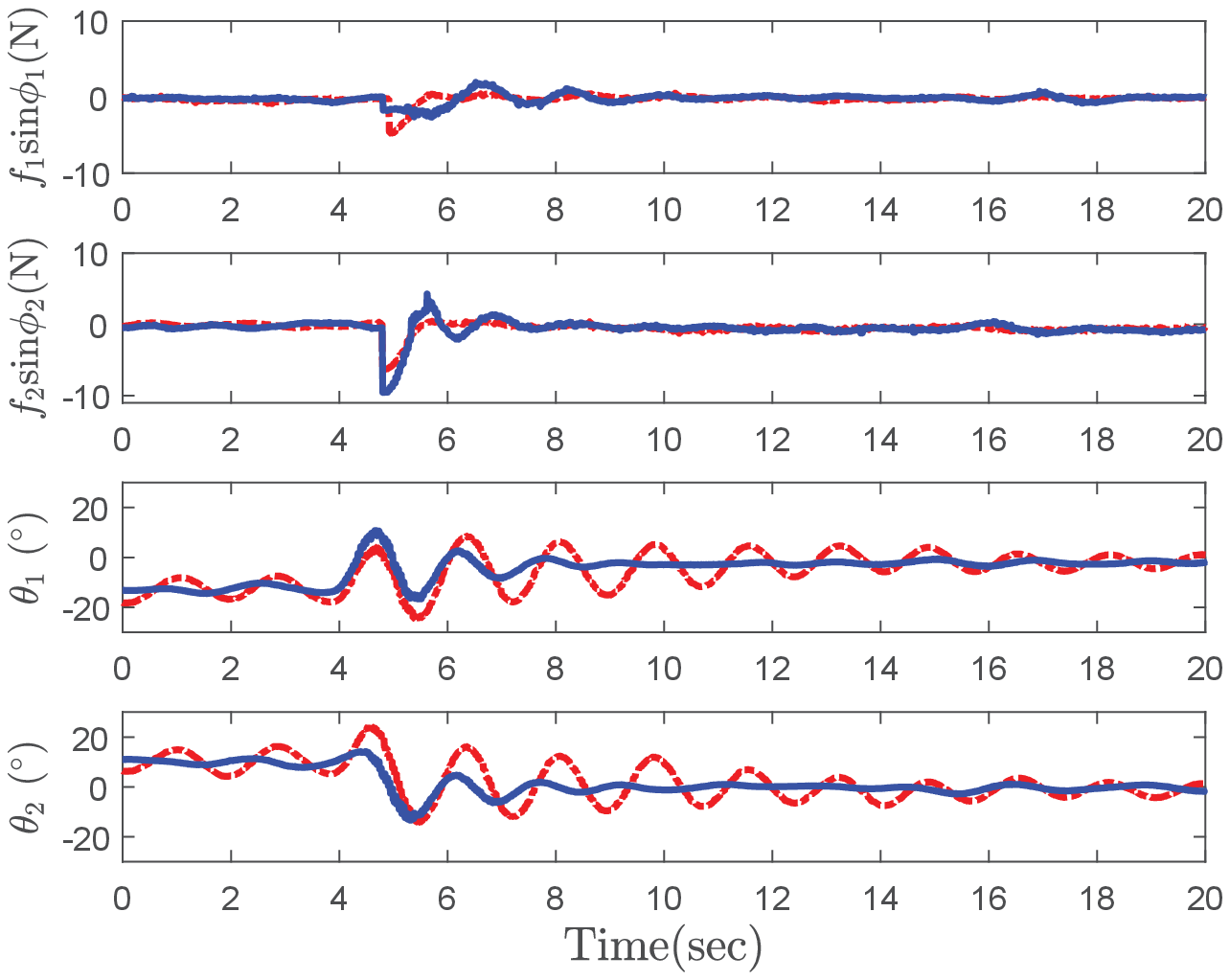}}\hspace{0.0mm}
	\caption{Results for Experiment 1-Test 1: Swing Elimination in transportation process. (black dashed lines: desired positions. red dotted-dashed lines: PD controller. blue solid lines: proposed controller.). }
	\label{figure-exp21}
\end{figure}

\begin{figure}[!htb]
	\centering
	\clearcaptionsetup{figure}
	\clearcaptionsetup{subfloat}
	\captionsetup[subfloat]{labelsep=period}
	\subfloat[ Positions of the two drones for Exp 1-Test 2.]{\includegraphics[width=0.46\textwidth]{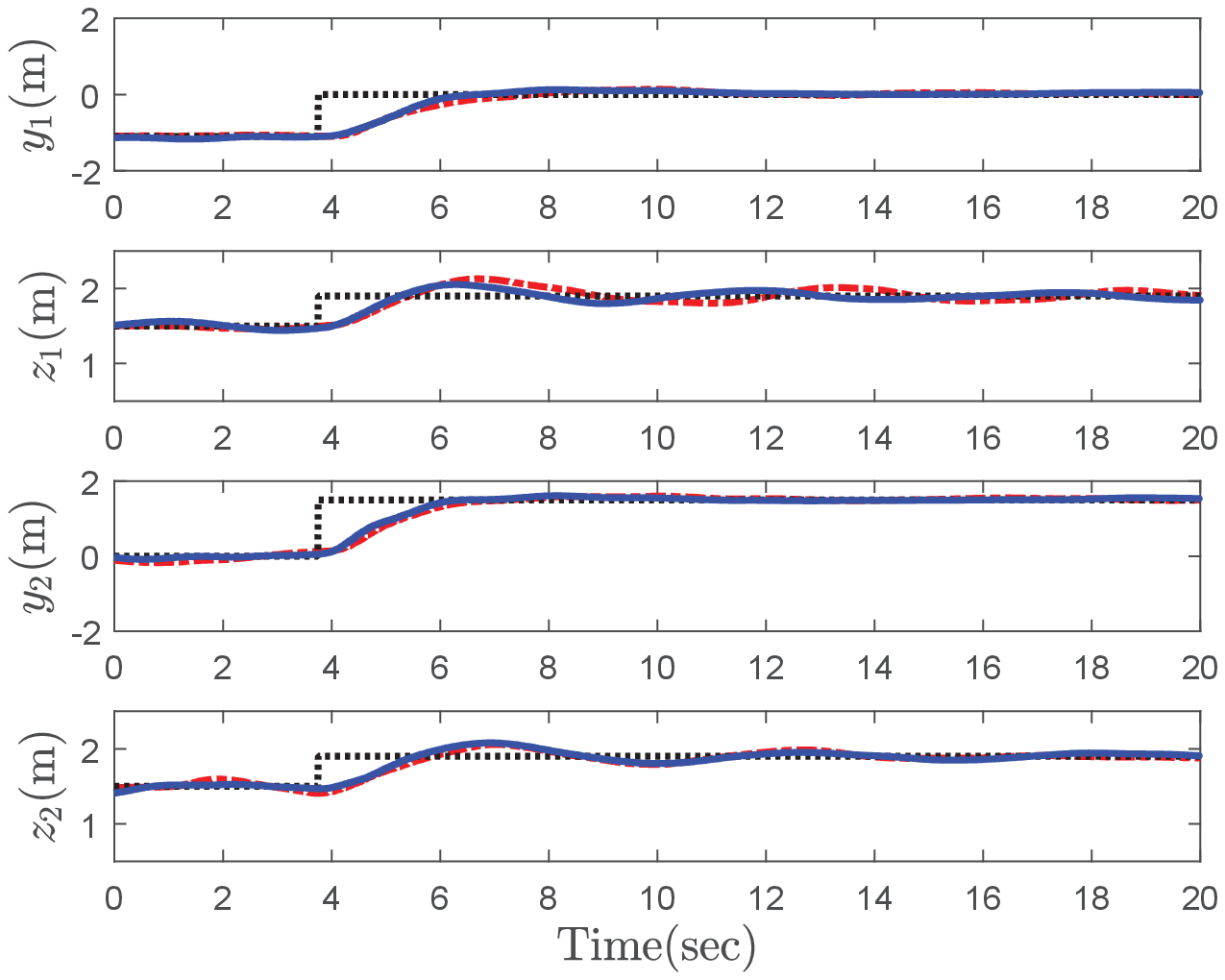}}\hspace{0.0mm}
	\captionsetup[subfloat]{labelsep=period}
	\subfloat[ Control input \& Swing angles.]{\includegraphics[width=0.46\textwidth]{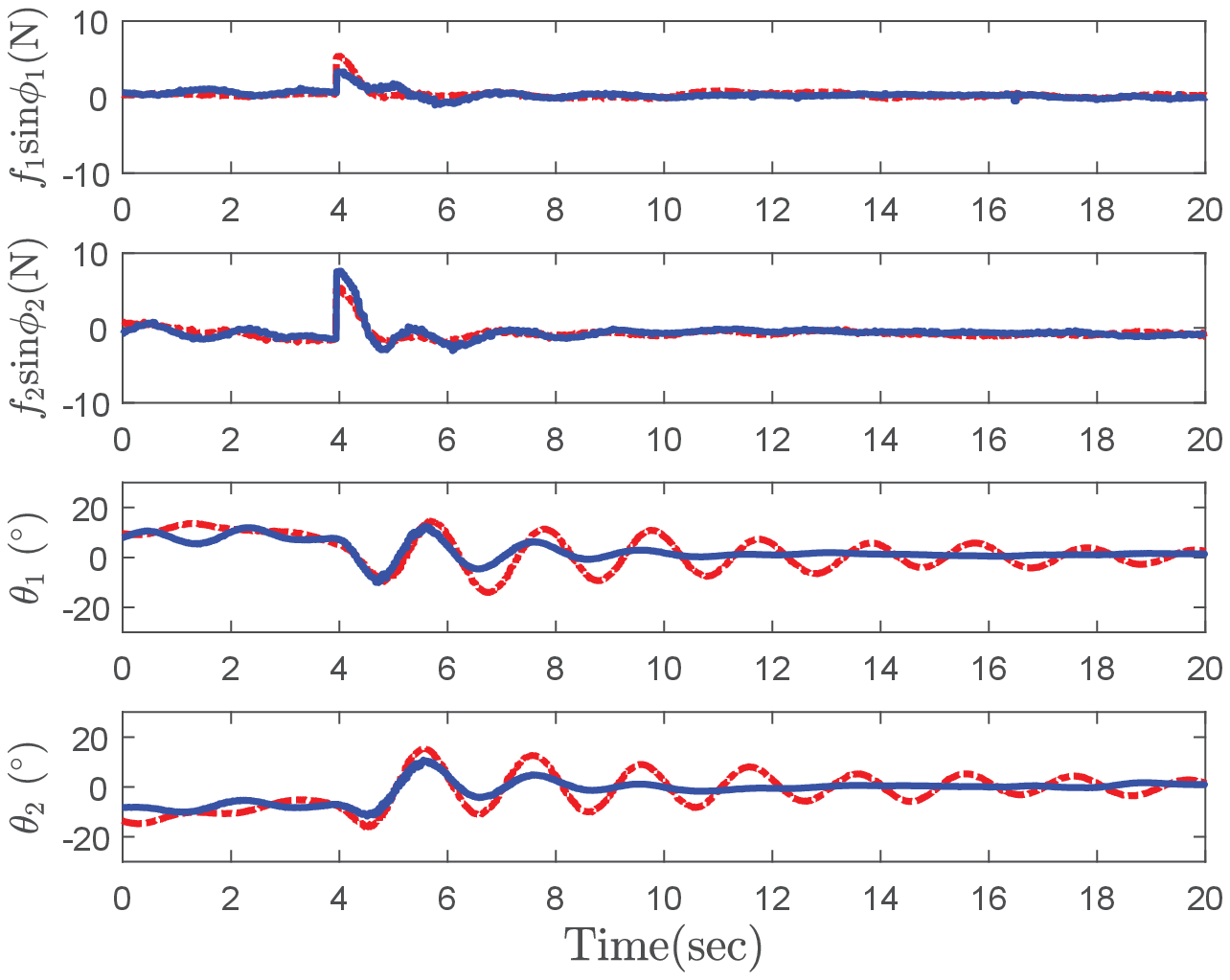}}\hspace{0.0mm}
	\caption{Results for Experiment 1-Test 2: Robustness against system uncertainties. (black dashed lines: desired positions. red dotted-dashed lines: PD controller. blue solid lines: proposed controller.).}
	\label{figure-exp22}
\end{figure}

{\bf Test 1: (Swing Elimination in transportation process).} 
It is seen from Table 1 that the initial and desired positions are set to verify whether the proposed method can meet the basic point-to-point transportation requirements. \textcolor[rgb]{0,0,1}{The maximum wind speed is up to $3\, \mathrm{m/s}$ measured by a digital anemometer AS8336 in the center of the experiment site, and the wind is in the negative direction of the Y-axis. In the actual experiment, the bar will continuously swing under the action of wind, which is particularly apparent in the PD controller, making the load's swing away from zero naturally. }

{\bf Test 2: (Robustness against system uncertainties).} 
To further verify the robustness of the proposed control schemes against parameter uncertainties, the bar's physical parameters and ropes' length  are changed to $m_{3} = 0.5\,{\rm kg},a = 1.5\,{\rm m}, l_{1}=l_{2} = 1.2\,{\rm m}$, yet the other system parameters are still the same with Test 1. The initial and desired positions are shown in Table 1. \textcolor[rgb]{0,0,1}{Like Test 1, this test is also conducted under wind disturbances.}

{\bf Test 3: (Robustness against external disturbances).} 
In the state of drones hovering, we test the anti-swing performance of the proposed method by manually adding the disturbance. The bar is disturbed purposely every $20$ seconds from different directions.
\begin{figure}[!htb]
	\centering
	\clearcaptionsetup{figure}
	\clearcaptionsetup{subfloat}
	\captionsetup[subfloat]{labelsep=period}
	\subfloat[ Positions of the two drones for Exp 1-Test 3.]{\includegraphics[width=0.46\textwidth]{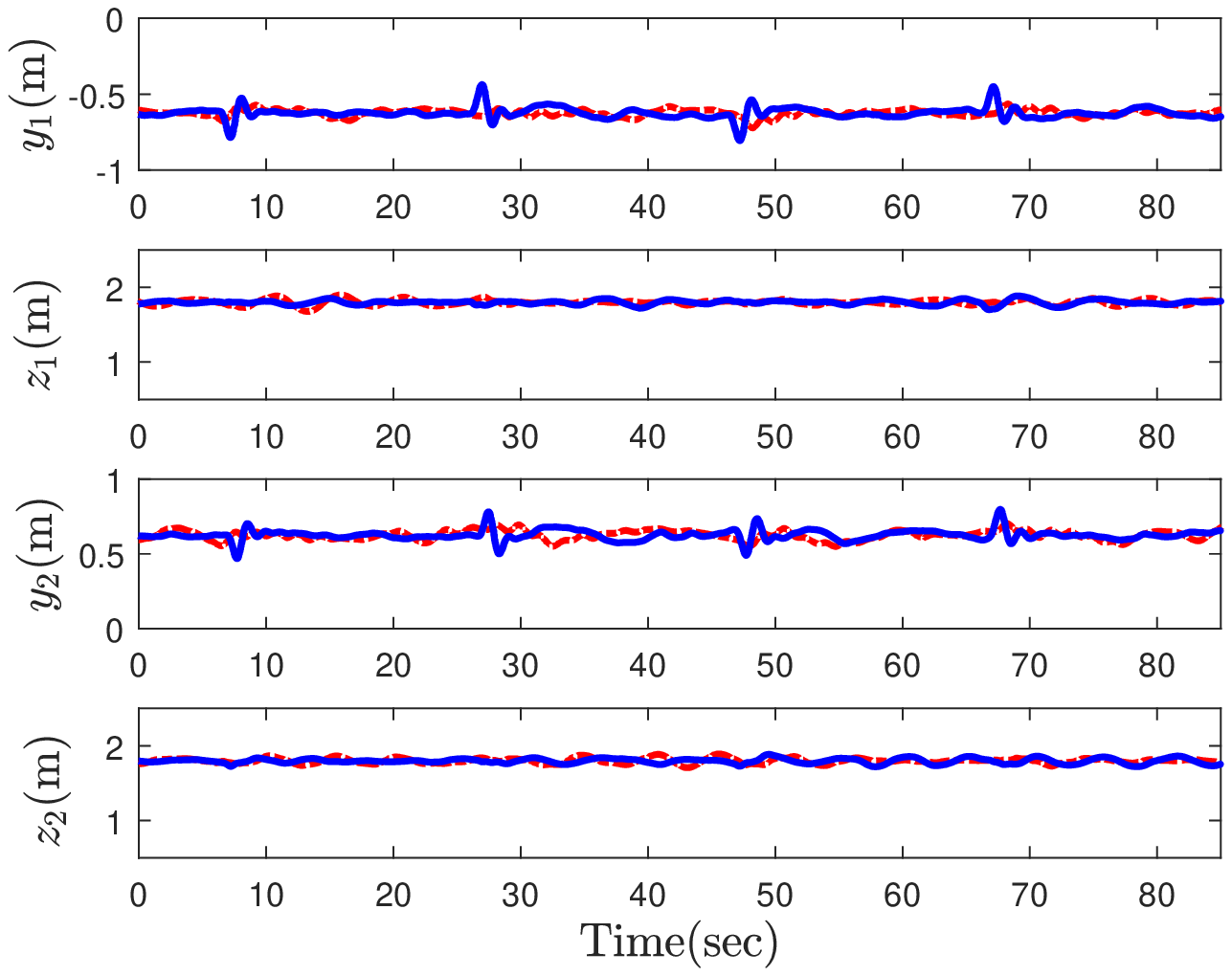}}\hspace{0.0mm}
	\captionsetup[subfloat]{labelsep=period}
	\subfloat[Control input \& Swing angles.]{\includegraphics[width=0.46\textwidth]{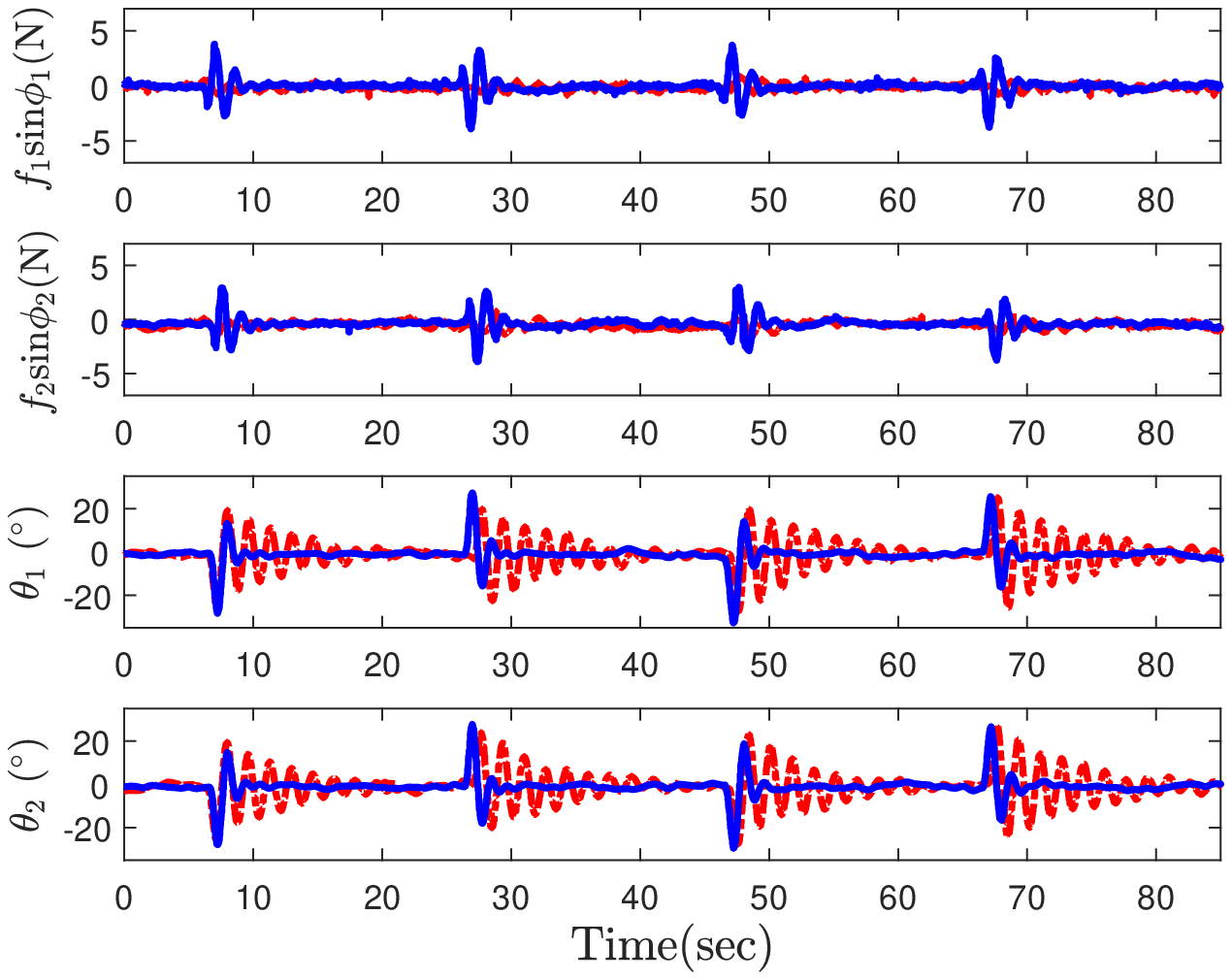}}\hspace{0.0mm}
	\caption{Results for Experiment 1-Test 3: Robustness against external disturbances. (red dotted-dashed lines: PD controller. blue solid lines: proposed controller.).}
	\label{figure-exp23}
\end{figure}
The experimental results of the three tests are provided in Fig. \ref{figure-exp21}-\ref{figure-exp23}, specifically. It is worth noting that the forces of the two control methods in the vertical direction are almost the same. Thus no comparison is shown here to save space. For Test 1, it is seen from Fig. \ref{figure-exp21} that the settling time of the proposed method is about only $20\%$ of that of the PD controller, which indicates the tremendous advantages of the proposed method in transient performance. For Test 2, one can find from Fig. \ref{figure-exp22} that the curves show a better anti-swing performance of the proposed control by comparing it with the classic PD controller, which is similar to the conclusion obtained in Test 1. For Test 3, we can observe from Fig. \ref{figure-exp23} that disturbances manually added from different directions are effectively rejected by the proposed control scheme. Further, when the disturbance is added, the control inputs $f_{1}{\rm sin}\phi_{1},f_{2}{\rm sin}\phi_{2}$ produce significant change to suppress the bar swing, which makes the convergence speed of $\theta_{1},\theta_{2}$ faster than that of PD controller. 

In summary, both controllers have achieved satisfactory drones positioning results, yet, the proposed control scheme (\ref{input1})-(\ref{input4}) presents better swing suppression effect and good robustness concerning system uncertainties and external disturbances such as wind disturbances and manually interference.

\subsection{Experiment Group 2}
\begin{table*}[!htb]
	\renewcommand\arraystretch{1.8}
	\centering
	\caption{Parameters of Test 1 in Experiment 2}
	\begin{tabular}{ccccc}
		\hline \hline
		{\bf{Experiment 2-Test 1}}$\quad$    & {\bf{Initial positions}}${\rm (m)}$ & {\bf{1st change}}${\rm (m)}$  & {\bf{2nd change}}${\rm (m)}$  & {\bf{3rd change}}${\rm (m)}$  \\
		\hline
		{$[y_{1},z_{1}]$} & $[-0.6,1.7]$  & $[-0.4,1.7] $&$ [-0.8,1.7] $&$ [-0.6,1.7]$\\
		{$[y_{2},z_{2}]$} & $[\,\,\,\,0.6,1.7]$  & $[\,\,\,\,0.4,1.7] $&  $[\,\,\,\,0.8,1.7]$& $[\,\,\,\,0.6,1.7]$\\
		\hline
		\hline
	\end{tabular}
\end{table*}
This group of experiments are performed to verify the control effect of the last term $\frac{\sigma\!\rho e_{y}}{\left( \rho \!-\!e_y^2\right) ^2 }$ in the proposed control inputs (\ref{input1})-(\ref{input2}). In addition, we implement some comparative experiments with the classic PD control as the comparative controller, whose control gains are set as $k_{p_{1}} = k_{p_{2}} = 5.2, k_{d_{1}}=k_{d_{2}} = 6.0,k_{p_{3}} = k_{p_{4}} = 6.0,  k_{d_{3}}=k_{d_{4}} = 8.0$ after sufficient tuning. Here, two comparative tests are carried out:

{\bf Test 1: (Relative distance test).} 
In the horizontal direction, the desired positions of the two drones are changed every $10$ seconds, which is recorded in Table $2$.
\begin{figure}[!htb]
	\centering
	{\includegraphics[width=0.46\textwidth]{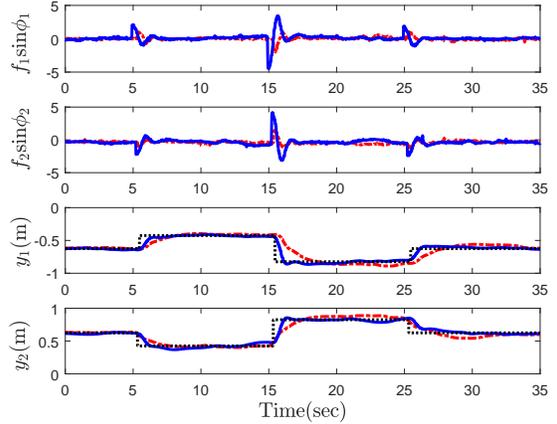}}\hspace{0.0mm}
	\caption{Results for Experiment 2-Test 1: Relative distance test. [(black dashed lines: desired positions. red dotted-dashed lines: PD controller. blue solid lines: proposed controller.)]}
	\label{figure-exp31}
\end{figure}

\begin{figure}[!htb]
	\centering
	{\includegraphics[width=0.46\textwidth]{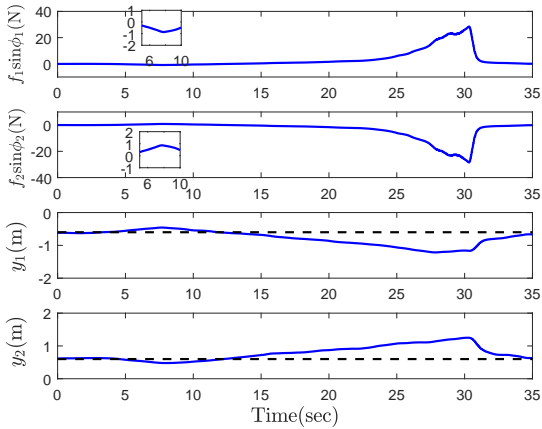}}\hspace{0.0mm}
	\caption{Results for Experiment 2-Test 2: Robustness against external disturbances. [(black dashed lines: desired positions. blue solid lines: proposed controller.)]}
	\label{figure-exp32}
\end{figure}
{\bf Test 2: (Response to the external disturbances brought by the last term).} 
In this test, external disturbance is added to make the two drones closer to or away from each other. The change of $f_{1}{\rm sin}\phi_{1},f_{2}{\rm sin}\phi_{2}$ is recorded to analyze the effectiveness of the proposed control strategy (\ref{input1})-(\ref{input2}).

For the two tests in Experiment 2, system parameters and control gains of the proposed methods are set the same as those in Experiment 1-Test 1. The forces and drones' positions of the two methods in the vertical direction are almost the same. Thus no comparison is shown here to save space. For Test 1, it is seen from Fig. \ref{figure-exp31} that both the proposed control scheme and classic PD method achieve satisfactory positioning control, which is the same as the conclusion of Experiment 1. However, the proposed method is more sensitive to the distance between the drones, which enables it to implement more efficient control for the distance. When the desired positions are changed, the error $e_{y} = e_{y_{1}} - e_{y_{2}}$ will cause the term $\frac{\sigma\!\rho e_{y}^2}{\left( \rho \!-\!e_y^2\right) ^2 }$ to generate a large momentary thrust, which pushes the drones back to the desired positions more quickly than PD method. During the return process, the term $\frac{\sigma\!\rho e_{y}^2}{\left( \rho \!-\!e_y^2\right) ^2 }$ continues to decay until it reaches zero. For Test 2, we constantly decrease and then increase the distance between the two drones in the horizontal direction and record the curves of $f_{1}{\rm sin}\phi_{1},f_{2}{\rm sin}\phi_{2}$. As Fig. \ref{figure-exp32} shows, in the process of increasing the distance, there will be a moment when the force is large enough to pull the two drones back to the desired positions. Considering that the rotating speed of propellers has a physical upper limit, we thus choose larger $\rho$ to avoid saturation. 

\section{Conclusion}
\label{conclusion}
\textcolor[rgb]{0,0,1}{
For the double drone-bar system, this paper focuses on the elimination of the bar-load swing angles and the coordination transportation. Specifically, a precise model set up by the Lagrangian modeling method facilitates the subsequent research. Besides, the proposed coordination controller successfully stabilizes the whole system and guarantees the asymptotic stability of the desired equilibrium point without any linearization or approximations. Many experiments are carried out indoors to verify the effectiveness and robustness of the controller. As the results show, the proposed method has better performance in swing suppression and coordination transportation. In future work, the experimental scene will be extended to the outdoor environment to solve more realistic problems encountered in transportation.}



\bibliographystyle{unsrt}        
\bibliography{autosam}           



\appendix   
\textcolor[rgb]{0,0,1}{
\section{Proof of Lemma 1}
It is learned that $\dot\theta_{1}=\dot\theta_{2}=\dot\theta_{3}=0$ is one of the solutions for (\ref{key-case1}). To show that it is the unique solution, we rewrite (\ref{key-case1}) into the following matrix form:
\begin{align}
A\cdot \left[\dot{\theta}_{1},\,\dot{\theta}_{2},\,\dot{\theta}_{3}\right]^T = 0,
\end{align}
where the coefficient matrix $A\in \mathbb{R}^{3\times 3}$ is defined as:
\begin{align}\label{coefficient matrix}
A = 
\left[                 
\begin{array}{ccc}
({{m_2}+{m_3}}) {l_1}{C_1} & {m_2}{l_2}{C_2} &  -(m_2+\frac{m_3}{2})a{S_3}\\
{l_1}{C_1} & {l_2}{C_2} & - a{S_3}\\ 
{l_1}{S_1} & - {l_2}{S_2} & a{C_3}
\end{array}
\right].
\end{align}
After some calculation, the determinant of $A$ is obtained as:
\begin{align}
\label{rank2}
|A| = \frac{m_3}{2}l_{1}l_{2}a\cdot(C_{1}C_{2+3}+C_{2}C_{1-3}).
\end{align}
According to Assumption 1, the range of the signals $\theta_{1},\theta_{2},\theta_{3}$ is $(-\frac{\pi}{2},\frac{\pi}{2})$. When $|\theta_2+\theta_3| = \frac{\pi}{2}$, the bar and the right rope will be in a straight line(see Fig.\ref{fig:2WMR}), which leads to an apparent contradiction. In addition, due to the effect of gravity, it is impossible that $|\theta_2+\theta_3| > \frac{\pi}{2}$, thus $C_{1}C_{2+3} > 0$. With some similar analysis, one further derives that $C_{1}C_{1-3} > 0$. Therefore, one can conclude that under the circumstance of Assumption 1, $|A| > 0$, which implies that $\dot\theta_{1}=\dot\theta_{2}=\dot\theta_{3}=0$ is the only solution for (\ref{key-case1}).}

\section{Proof of Lemma 2}

It is learned from (\ref{error signals}) that $e_{z_{2}} = z_{2} - z_{2d}$. Subsequently, it will be demonstrated with reduction to absurdity that $e_{z_{2}} = 0$ always holds for (\ref{case3}). This part of ananlysis is split into the following three cases.

\subsection{Suppose $e_{z_{2}} = 0$}
If $e_{z_{2}} = 0$, it can be concluded from (\ref{case3}) that 
\begin{equation*}
\tan\theta_{1}= \tan\theta_{2},\,\,\tan\theta_{3} = 0.\eqno{(A.1)}
\end{equation*}
According to Assumption 1, the range of the signals $\theta_{1},\theta_{2},\theta_{3}$ is $(-\frac{\pi}{2},\frac{\pi}{2})$. Therefore, one always has $\theta_{1} = \theta_{2},\theta_{3} = 0$ for $e_{z_{2}} = 0$.

\subsection{Suppose $e_{z_{2}} > 0$}
In section $2$, the drones' positions are denoted by $\boldsymbol\xi_{1}(t) = [y_1(t),z_1(t)]^{T} = [y,z]^{T}$ and $\boldsymbol\xi_{2}(t) = [y_{2}(t),z_{2}(t)]^{T}=[y+l_{1}S_{1}+l_{2}S_{2}+aC_{3},z-l_{1}C_{1}+l_{2}C_{2}+aS_{3}]^{T}$. If $e_{z_{2}} > 0$, it can be concluded from (\ref{relation0}) and (\ref{relation1}) that 
\begin{equation*}
\label{B_unequal}
\begin{aligned}
&{e_{{z_1}}} = -\frac{k_{p_4}}{k_{p_3}}{e_{{z_2}}} ,\,z_{1d} = z_{2d},\, l_{1} = l_{2}=l\\
\Rightarrow&\,\,e_{z_{1}} < 0,\,z_{2}-z_{1}>0\\
\Rightarrow&\,\,l(C_{2}-C_{1})+aS_{3}>0.
\end{aligned}\eqno{(A.2)}
\end{equation*}
Under the condition of $e_{z_{2}} > 0$ and $-\frac{\pi}{2}<\theta_{1},\theta_{2},\theta_{3}<\frac{\pi}{2}$, one may deduce from (\ref{case3}) that
\begin{equation*}
\begin{aligned}
&|\tan {\theta _1}|\cdot| {\frac{1}{2}{m_3}g + k_{p_4}{e_{{z_2}}}} |=|\tan {\theta _2}|\cdot| {\frac{1}{2}{m_3}g - k_{p_4}{e_{{z_2}}}} |\\
\Rightarrow&\,\,|\tan {\theta _1}|<|\tan {\theta _2}|\,\,\Rightarrow\,\,|\theta_{1}|<|\theta_2|\\
\Rightarrow&\,\,C_1>C_2,
\end{aligned}\eqno{(A.3)}
\end{equation*}
which further leads to 
\begin{equation*}		\label{theta3}
\begin{aligned}
&l(C_{2}-C_{1})+aS_{3}>0,\,C_1>C_2\\
\Rightarrow&\,\,aS_{3}>l(C_{1}-C_{2})>0\\
\Rightarrow&\,\,\theta_{3}>0.
\end{aligned}\eqno{(A.4)}
\end{equation*}
It is concluded from (\ref{case3}) that 
\begin{equation*}
\begin{aligned}
\label{sinphi1}
&- k_{p_4}{e_{{z_2}}} = {f_2}\sin {\phi _2}\tan {\theta _3}\\
\Rightarrow&\,\,{f_2}\sin {\phi _2}\tan {\theta _3}<0,\,\Rightarrow\,{f_2}\sin {\phi _2}<0\\
\Rightarrow&\,\,\tan {\theta _1}( {\frac{1}{2}{m_3}g + k_{p_4}{e_{{z_2}}}})<0\\
\Rightarrow&\,\,\tan {\theta _1}<0,\,\Rightarrow\, \theta_{1}<0.
\end{aligned}\eqno{(A.5)}
\end{equation*}
where (A.4) is utilized. Regarding (A.5), one further makes the following deductions according to (\ref{relation0}), (\ref{input1}), (\ref{dotv=0}), (\ref{key-Q1}) and (\ref{key-inte2}) that
\begin{equation*}
\begin{aligned}\label{ey1}
&-{f_2}\sin {\phi _2}={f_1}\sin{\phi _1} =  - k_{p_1}{e_{{y_1}}} - \frac{\sigma\!\rho e_{y}}{\left( \rho \!-\!e_y^2\right) ^2 }\,>\,0\\
\Rightarrow&\,\,{f_1}\sin{\phi _1} = -{e_{{y_1}}}\bigg(k_{p_1} +\frac{\sigma\!\rho (1+\frac{k_{p_{1}}}{k_{p_{2}}})}{\left[ \rho \!-\!(e_{y_{1}}-e_{y_{2}})^2\right] ^2 }  \bigg)\,>\,0\\
\Rightarrow&\,\,e_{y_{1}}\,<\,0 \,\Rightarrow\, e_{y_{2}}\,>\,0.
\end{aligned}\eqno{(A.6)}
\end{equation*}
Furthermore, because of the restrictions on $f_{1}\cos\phi_{1}$, $f_{2}\cos\phi_{2}$ in (\ref{relation2}), it is inferred from (\ref{input3}), (\ref{dotv=0}), (\ref{relation1}) and (A.6) that
\begin{equation*}
\begin{aligned}
\label{12m3g}
&{m_1}g < {f_1}\cos {\phi _1} < \left( {{m_1}{ + }{m_3}} \right)g\\
\Rightarrow&\,\,- k_{p_3}{e_{{z_1}}} + ( {{m_2} + \frac{1}{2}{m_3}} )g\,<\, \left( {{m_1}{ + }{m_3}} \right)g\\
\Rightarrow&\,\,- k_{p_3}{e_{{z_1}}}<\frac{1}{2}m_{3}g\Rightarrow\,  k_{p_4}{e_{{z_2}}}<\frac{1}{2}m_{3}g\\
\Rightarrow&\,\,\frac{1}{2}m_{3}g - k_{p_4}{e_{{z_2}}} > 0.
\end{aligned}\eqno{(A.7)}
\end{equation*}
Based on the results in (A.5) and (A.7), the following conclusion is derived from (\ref{case3}) that
\begin{equation*}
\begin{aligned}
\label{theta2}
&{f_2}\sin {\phi _2} <0,\,\frac{1}{2}m_{3}g - k_{p_4}{e_{{z_2}}} > 0,\\
&\tan {\theta _2}( {\frac{1}{2}{m_3}g - k_{p_4}{e_{{z_2}}}} t) = {f_2}\sin {\phi _2}\\
\Rightarrow&\,\,\tan {\theta _2}<0\,\Rightarrow\,{\theta _2}<0.
\end{aligned}\eqno{(A.8)}
\end{equation*}

According to the analysis from (A.2) to (A.8), suppose $e_{z_{2}} > 0$, one can obtain that $\theta_{1}<0,\theta_{2}<0,e_{y_{1}}<0,e_{y_{2}}>0$. When $e_{y_{1}}<0,e_{y_{2}}>0$, the distance between the two drones in horizontal direction can be derived as:
\begin{equation*}
\begin{aligned}
\label{conclu1}
&e_{y_{2}}-e_{y_{1}} = (y_{2}-y_{1}) - (y_{2d}-y_{1d})>0\\
\Rightarrow&\,\,y_{2}-y_{1}>y_{2d}-y_{1d}\,\Rightarrow\,y_{2}-y_{1}>a.
\end{aligned}\eqno{(A.9)}
\end{equation*}
When $\theta_{1}<0,\theta_{2}<0$, the distance between the two drones in the horizontal direction can be derived as:
\begin{equation*}
\begin{aligned}
\label{conclu2}
&y_{2}-y_{1} = l_{1}S_{1}+l_{2}S_{2}+aC_{3}<aC_{3}-|l_{1}S_{1}|-|l_{2}S_{2}|\\
\Rightarrow&\,\, y_{2}-y_{1} < a,
\end{aligned}\eqno{(A.10)}
\end{equation*}
where $\boldsymbol\xi_{1}(t) = [y_1(t),z_1(t)]^{T} = [y,z]^{T}$ and $\boldsymbol\xi_{2}(t) = [y_{2}(t),z_{2}(t)]^{T}=[y+l_{1}S_{1}+l_{2}S_{2}+aC_{3},z-l_{1}C_{1}+l_{2}C_{2}+aS_{3}]^{T}$ are utilized. It is obvious that (A.9) and (A.10) are contradictory, thus $e_{z_{2}} > 0$ does not exist.

\subsection{Suppose $e_{z_{2}} < 0$} 
The proof steps for $e_{z_{2}} < 0$ and $e_{z_{2}} > 0$ are roughly the same, thus similar deductions can be performed when $e_{z_{2}} < 0$. It is impossible that $e_{z_{2}} < 0$ as well.

According to the above analysis, the only solution for (\ref{case3}) is $e_{z_{2}}=0$.        

%

\end{document}